\documentclass[a4paper,12pt,reqno]{amsart}
\usepackage{mathrsfs,latexsym, bm, amscd,amssymb,url}
\usepackage{xcolor, mathtools}
\PassOptionsToPackage{pdfusetitle,pagebackref,colorlinks}{hyperref}
\usepackage{bookmark}
\hypersetup{
  linkcolor={red!70!black},
  citecolor={green!70!black},
  urlcolor={teal!80!black},
	linkcolor={teal!80!black}
}
\usepackage[all,cmtip]{xy}  
\usepackage{graphicx}
\usepackage{color}
\usepackage{hyperref}
\usepackage{enumerate}
\usepackage{stmaryrd}

\usepackage[T1]{fontenc}         
\usepackage{ae}									
\usepackage[all,cmtip]{xy}  
\usepackage{graphicx}
\usepackage{color}

\newtheorem{theorem}{Theorem}[section]

\newtheorem{lemma}[theorem]{Lemma}
\newtheorem{proposition}[theorem]{Proposition}
\newtheorem{corollary}[theorem]{Corollary}

\newtheorem{conjecture}[theorem]{Conjecture}

\theoremstyle{question}
\newtheorem{question}[theorem]{Question}
\theoremstyle{definition}
\newtheorem{definition}[theorem]{Definition}

\theoremstyle{remark}
\newtheorem{remark}[theorem]{Remark}

\newtheoremstyle{cited}{.5\baselineskip\@plus.2\baselineskip\@minus.2\baselineskip}{.5\baselineskip\@plus.2\baselineskip\@minus.2\baselineskip}{\itshape}{}{\bfseries}{\bfseries .}{5pt plus 1pt minus 1pt}{\thmname{#1}\thmnumber{~#2}\thmnote{ \normalfont#3}}
\theoremstyle{cited}

\newtheoremstyle{citeddef}{.5\baselineskip\@plus.2\baselineskip\@minus.2\baselineskip}{.5\baselineskip\@plus.2\baselineskip\@minus.2\baselineskip}{}{}{\bfseries}{\bfseries .}{5pt plus 1pt minus 1pt}{\thmname{#1}\thmnumber{~#2}\thmnote{ \normalfont#3}}
\theoremstyle{citeddef}

\newcommand{\Z}{\mathbb Z}
\newcommand{\Q}{\mathbb Q}

\newcommand{\C}{\mathbb C}
\newcommand{\N}{\mathbb N}
\newcommand{\R}{\mathbb R}

\newcommand{\reso}{_{\text{res}}}

\newcommand{\ba}{{}_{\bm{a}}}

\newcommand{\baa}{{}_{\bm{a}'}}
\newcommand{\bfa}{\bm{a}}
\newcommand{\diae}{{}^\diamond\!E} 
\newcommand{\diaf}{{}^\diamond\!F}

\newcommand{\Sh}{\operatorname{Sh}}

\newcommand{\im}{\operatorname{im}}

\newcommand{\Hom}{\operatorname{Hom}}

\newcommand{\GL}{\operatorname{GL}}
\newcommand{\SL}{\operatorname{SL}}

\newcommand{\Sing}{\operatorname{Sing}}

\newcommand{\rank}{\operatorname{rank}}

\newcommand{\Gr}{\operatorname{Gr}}

\newcommand{\pc}{\operatorname{para-c}}
\newcommand{\pd}{\operatorname{para-deg}}
\newcommand{\MB}{\operatorname{M_B}}
\newcommand{\pe}{\operatorname{\pi^{\text{\'et}}_1}}

\newcommand{\dashedlongrightarrow}{\xymatrix@1@=15pt{\ar@{-->}[r]&}}
\renewcommand{\longrightarrow}{\xymatrix@1@=15pt{\ar[r]&}}
\renewcommand{\mapsto}{\xymatrix@1@=15pt{\ar@{|->}[r]&}}
\renewcommand{\twoheadrightarrow}{\xymatrix@1@=15pt{\ar@{->>}[r]&}}
\newcommand{\hooklongrightarrow}{\xymatrix@1@=15pt{\ar@{^(->}[r]&}}
\newcommand{\congpf}{\xymatrix@1@=15pt{\ar[r]^-\sim&}}
\renewcommand{\cong}{\simeq}
\usepackage{tikz-cd}
\usepackage{enumitem}

\PassOptionsToPackage{pdfusetitle,pagebackref,colorlinks}{hyperref}
\usepackage{bookmark}
\hypersetup{
  linkcolor={red!70!black},
  citecolor={green!70!black},
  urlcolor={blue!80!black}
}

\newcommand{\scrC}{\mathscr{C}}

\begin{document}  
\title[$\pi_1$-small divisors and fundamental groups of varieties]{$\pi_1$-small divisors and fundamental groups of varieties}

\author{Feng Hao}

\address{Department of Mathematics, KU Leuven, Celestijnenlaan 200b - box 2400, 3001 Leuven, Belgium.}
\email{feng.hao@kuleuven.be}
\date{\today}
%\date{October 10, 2017; 
%\copyright{\ Stefan Schreieder 2017}}
%
% 
% 14C20: Divisors, linear systems, invertible sheaves
% 14D07: Variation of Hodge structures
% 14F35: Homotopy theory; fundamental groups
%  14J29: Surfaces of general type


\begin{abstract} Lasell and Ramachandran show that the existence of rational curves of positive self-intersection on a smooth projective surface $X$ implies that all the finite dimensional linear representations of the fundamental group $\pi_1(X)$ are finite. In this article, we generalize Lasell and Ramachandran's result to the case of $\pi_1$-small divisors on quasiprojective varieties. We also study $\pi_1$-small curves and hyperbolicity properties of smooth projective surfaces of general type with infinite fundamental groups.
\end{abstract}

\maketitle

\section{Introduction}\label{intro}

Lefschetz hyperplane theorem indicates that (possibly singular) curves of positive self-intersection on a smooth projective surface contains much information of the fundamental group of the projective surface. More generally, Napier and Ramachandran \cite{NR98} showed that the fundamental group of a local complete intersection subvariety with ample normal bundle dominates a subgroup of finite index in the fundamental group of the ambient variety. In principle, one should expect that the smooth model of the singular hyperplane also captures some information of the fundamental group of the ambient space. In particular, Nori asked the following natural question while studying the weak Lefschetz theorem in \cite{Nori83}.
\begin{question}\label{Quest:Nori} Let $X$ be a smooth projective surface. Suppose there is a (possibly singular) rational curve $C$ on $X$ with the self-intersection $C^2>0$. Is the fundamental group $\pi_1(X)$ finite?
\end{question} 

When $C$ is smooth, the positive answer of this question follows directly from the Lefschetz hyperplane theorem. Nori studied this question for nodal curves with large enough self-intersections. Gurjar related this question to the Shafarevich uniformization conjecture. Moreover, \cite{GP09}, %\cite{GP18}, 
\cite{GPP11}, etc., answered this question in some special cases of surfaces for which the Shafarevich conjecture holds. The following theorem due to Lasell and Ramachandran \cite{LR96} answers the above question when $\pi_1(X)$ is linear.

\begin{theorem}[Lasell-Ramachandran]\label{Thm:L-R}
Let $X$ be a smooth projective surface. Let the effective divisor $B=\sum_iC_i$ be a finite union of (possibly singular) rational curves $C_i$ on $X$. Assume there is a divisor $D$ supported on $B$ with $D^2>0$, then for any finite dimensional linear representation $\rho\colon \pi_1(X)\to \GL(m, \C)$, $\rho(\pi_1(X))$ is finite.
\end{theorem}

\begin{remark}
(a) To completely answer Question \ref{Quest:Nori}, it is natural to ask whether there exists a smooth projective surface $X$ such that the fundamental group $\pi_1(X)$ is infinite, but any finite dimensional linear representation of $\pi_1(X)$ has finite image. This is still an open question. In fact, there is a finitely presented residually finite infinite group $G$ such that representations $\rho\colon G\to \GL(m, \C)$ have finite images for all $m$. However, it is not clear whether or not $G$ is the fundamental group of a smooth projective variety (see e.g., \cite[Remark 0.2 (5)]{BKT13}).

(b) Theorem \ref{Thm:L-R} is extended by Zuo \cite{Zuo96a} to the case that $D^2=0$.
\end{remark}

In this article, we generalize the above theorem to $\pi_1$-small divisors (see Definition \ref{def:pi1small}) on higher dimensional quasi-projective varieties. The following main theorem shows that a $\pi_1$-small divisor with some positivity properties captures lots of information of the fundamental group of the ambient variety. We fix the following notations throughout this article. 

Let $X$ be a complex smooth projective variety of dimension $n\geq2$, $D=\sum_{k=1}^sD_k$ be a simple normal crossing divisor on $X$, and  $Y=\sum_{i=1}^r n_iY_i$ be a divisor on $X$. Denote the quasiprojective variety $X-D$ by $U$. Also, denote the restriction of $Y_i$ on $U$ by $Y^o_i$ for each $i$, and $Y^o\coloneqq\sum_{i=1}^r n_iY_i^o$. We say that $Y$ \textsl{intersects mildly with respect to} $D$, if $\dim Y_i\cap D_p\cap D_q\leq n-3$ for any $i$ and $p\neq q$. With these notations, we have the following theorem.

\begin{theorem}\label{Thm:main}
Let $X$ be a smooth projective variety of dimension $n\geq2$, together with a normal crossing divisor $D$ and a divisor $Y=\sum_{i=1}^r n_iY_i$. Suppose that 
\begin{enumerate}
\item \label{item:main1} $Y^o$ is a $\pi_1$-small divisor on $U$, i.e., the homomorphism $ \pi_1(\widehat{Y}^o_i)\to \pi_1(U)$ has finite image for each $i$ and some (hence for any) desingularization $\widehat{Y}^o_i$ of $Y^o_i$,

\item \label{item:main2} $Y$ intersects mildly with respect to $D$, and

\item \label{item:main3} there exists an ample divisor $H$ on $X$, such that $Y^2\cdot H^{n-2}>0$.
\end{enumerate}
Then any finite dimensional linear representation $\rho\colon\pi_1(U)\to \GL(m, \C)$ has finite image. 
\end{theorem}

\begin{remark} (a) In fact, by the proof of Theorem \ref{Thm:main}, the theorem still holds if one replaces the assumption (\ref{item:main1}) by a ``weaker'' condition: the homomorphism \[\pi_1(\widehat{Y}^o_i)\to \pi_1(U)\to\pe(U)\] is trivial for each $i$ and some desingularization $\widehat{Y}^o_i$ of $Y^o_i$ (See Corollary \ref{cor:strong-main}).

(b) If the divisor $Y$ does not intersect with the boundary divisor $D$, or $D=\emptyset$ (see Corollary \ref{Cor:main-proj}), then assumption (\ref{item:main2}) of Theorem \ref{Thm:main} is satisfied automatically.

(c) The condition $Y^2\cdot H^{n-2}>0$, rather than  $Y^n>0$, is the right condition in assumption (\ref{item:main3}) of Theorem \ref{Thm:main}. In fact, if $X$ is the blowup of an arbitrary smooth projective threefold at a closed point, then the exceptional divisor $E$ is $\pi_1$-small and $E^3>0$. However, $E$ has no control on the fundamental group of $X$.

(d) For a normal variety $X$ with infinite fundamental group $\pi_1(X)$ and a $\pi_1$-small divisor $Y$ on $X$ admitting a positive dimensional linear system $|Y|$, by \cite[Theorem 1.12]{Kol95} or \cite[Proposition 1.4]{Camp91}, general irreducible members of $|Y|$ are not $\pi_1$-small if the linear system $|Y|$ has base points. Therefore, although the assumption (\ref{item:main3}) of Theorem \ref{Thm:main} provides some positivity properties to the $\pi_1$-small divisor $Y^o$, it is unclear a priori that $Y^o$ affects the fundamental group $\pi_1(U)$ much. 
\end{remark}

In the projective case, we have the following corollary of Theorem \ref{Thm:main}.
\begin{corollary}\label{Cor:main-proj}
Let $(X, \Delta)$ be a projective klt variety of dimension $n\geq2$ with a boundary divisor $\Delta$. Let $H$ be an ample divisor, and $Y$ be a $\pi_1$-small divisor such that $Y^2\cdot H^{n-2}>0$. Then any finite dimensional linear representation $\rho\colon\pi_1(X)\to \GL(m, \C)$ is finite. 
\end{corollary}
 
If a variety admits infinitely many $\pi_1$-small divisors with some negativity properties, then its fundamental groups can not be large.  
\begin{corollary}\label{Cor:infiniteneg}
Let $X$ be a smooth projective variety with a simple normal crossing divisor $D$ and Picard number $\rho$. Denote the complement $X-D$ by $U$. Suppose that there are more than $\rho^2+\rho+1$ prime divisors $\{Y_i\}_{i\in\Lambda}$ on $X$ such that for all $i\in\Lambda$
\begin{enumerate}
\item $Y^o_i\coloneqq Y_i\cap U$ is $\pi_1$-small on $U$,
\item $Y_i$ intersects mildly with respect to $D$, and

\item $Y_i^2\cdot H^{n-2}<0$ for an ample divisor $H$.
\end{enumerate}
Then any finite dimensional linear representation $\rho\colon\pi_1(U)\to \GL(m, \C)$ is finite.
\end{corollary}  

For the proof of Theorem \ref{Thm:main}, we follow closely the strategy of the original proof in Lasell and Ramachandran \cite{LR96} and the proof in Zuo \cite{Zuo99}. However, we need to pay more price  while dealing with some issues coming from the boundary. Applying the celebrated result -- the existence of harmonic maps associated to $p$-adic representations of fundamental groups of quasiprojective varieties -- due to Jost and Zuo \cite{JZ97}, one can show that the Betti moduli space of the quasiprojective variety $U$ is of dimension 0. In this step, one notices that either the spectral covering coming from the harmonic map given by Jost and Zuo is \'etale over $\pi_1$-trivial (see Definition \ref{def:pi1small}) divisors or the $\pi_1$-trivial divisors lie in the branched loci of the spectral covering. This observation helps to tackle problems raised from the boundary divisor. Once we get that the Betti moduli space of $U$ is of dimension 0, using some deformation and perturbation methods in Mochizuki \cite{Moc06} and \cite{AHL19}, we get that any semisimple representation of $\pi_1(U)$ underlies a complex variation of Hodge structures with quasi-unipotent monodromy along the boundary divisor. Via the covering trick due to Kawamata, we then reduce the proof to the case of a complex variation of Hodge structures with unipotent monodromy. Then using some semi-negativity results due to Zuo \cite{Zuo00}, we finally get that for a semisimple representation $\rho$ of $\pi_1(U)$, there is a finite index subgroup of $\pi_1(U)$ having finite image via the representation $\rho$. This implies that any representation of $\pi_1(U)$ is finite. \\

Now we turn to the study of $\pi_1$-small divisors on projective surfaces with infinite fundamental groups. We first have the following proposition regarding a special behaviour of configurations of $\pi_1$-small curves on surfaces. 

\begin{proposition}\label{C}
Let $X$ be a smooth projective surface with infinite fundamental group $\pi_1(X)$. Suppose there are infinitely many irreducible $\pi_1$-small curves $\{C_i\}_{i\in \Lambda}$ on $X$ such that $C_i^2>0$. Then either
\begin{enumerate}
\item \label{glf1}there exists a finite set $T=\{p_1, \ldots, p_l\}$ of closed points of $X$ such that $C_i\cap T\not=\emptyset$ for all but finitely many $i\in \Lambda$; Or  
\item \label{glf2}$X$ has generically large fundamental group. 
\end{enumerate} 
On the other hand, suppose there are infinitely many irreducible $\pi_1$-small curves $\{C_i\}_{i\in \Lambda}$ on $X$ such that $C_i^2<0$. Then $X$ has generically large fundamental group. 
\end{proposition}

Refer to Definition \ref{def:glfg} for the definition of generically large fundamental groups. In fact, if $X$ has generically large fundamental group, then $X$ admits at most countably many $\pi_1$-small curves.  In a somewhat different direction, we use the method in the proof of Proposition \ref{C} together with the known results mentioned below to study hyperbolicity properties of smooth projective surfaces of general type with fundamental groups admitting finite dimensional linear representations of infinite images. 

\begin{conjecture}[Geometric Lang's conjecture]
Let $X$ be a smooth projective variety of general type. Then the union of
all irreducible positive dimensional subvarieties of $X$ not of general type is a proper 
closed algebraic subset of $X$.
\end{conjecture}

Bogomolov \cite{Bog77} proved the geometric Lang's conjecture for surfaces of general type with $c^2_1(S)>c_2(S)$. In \cite{LM95}, Lu and Miyaoka showed that any surface of general type contains at most finitely many smooth rational curves and elliptic curves. Zuo \cite{Zuo96b} showed that any variety $X$ admitting a big Zariski dense representation to an almost simple algebraic group is Chern-hyperbolic. In particular, the geometric Lang's conjecture holds true for this kind of surfaces. The conjecture was proved for surfaces of general type with irregularity at least two by Lu \cite{Lu10} (See also \cite{LP12}). Even for surfaces, the geometric Lang's conjecture is widely open to the best of the authors' knowledge.

\begin{theorem}\label{Thm:main-hyperbolicity}
Let $X$ be a smooth projective surface of general type. Suppose there is a representation $\rho\colon\pi_1(X)\to \GL(m, \C)$ with infinite image for some $m\in \Z^+$. Then 
\begin{enumerate}
\item \label{Thm:rat-curve}$X$ contains at most finitely many (possibly singular) rational curves.

\item \label{Thm:ell-curve}Moreover, if the first Betti number $b_1(X)\not=1$, then $X$ contains at most finitely many (possibly singular) elliptic curves.
\end{enumerate}
\end{theorem}

The statement (\ref{Thm:rat-curve}) is a direct corollary of Theorem \ref{Thm:L-R} and a generalized Miyaoka-Yau inequality. For statement (\ref{Thm:ell-curve}), we use similar methods in the proof of Proposition \ref{C} to reduce the proof to the known case in Zuo \cite{Zuo96b}, in which he studied the hyperbolicity of the spectral covering coming from the pluriharmonic map associated to a $p$-adic representation of fundamental group of $X$. 

All known infinite K\"ahler groups have some finite-dimensional
linear representations with infinite images, so Theorem \ref{Thm:main-hyperbolicity} applies to all of them. The following natural question is raised from the proof of Theorem \ref{Thm:main-hyperbolicity}. In fact, if the answer to the following question is ``yes'', then Theorem \ref{Thm:main-hyperbolicity} is a direct corollary of Theorem \ref{Thm:main} and \cite{Lu10}.
\begin{question}
Let $X$ be a smooth projective variety with trivial first Betti number $b_1(X)=0$. Is any integral curve with geometric genus 1 on $X$ a $\pi_1$-small curve?
\end{question}
\subsection*{Notation}\begin{itemize}
\item Varieties $X$ are defined over the field of complex numbers $\mathbb{C}$.
\item The fundamental group $\pi_1(X)$ is $\pi_1(X, a)$ for an unspecified base point $a$ in $X$.
\item A divisor $D$ on a normal variety $X$ is a linear combination $D\coloneqq \sum_in_iD_i$, with $D_i$ being closed irreducible subvarieties of codimension one and $n_i\in \Z$.
\end{itemize}

\noindent\textbf{Acknowledgements}

The author would like to thank Donu Arapura, Nero Budur, Ya Deng, 
Claudio Llosa Isenrich, Yongqiang Liu, and Kang Zuo for very helpful communications. This work was supported by the Research Foundation Flanders (FWO) Grant ``Topology, birational geometry and vanishing theorem for complex algebraic varieties''[project no. 1280421N].

\section{Preliminaries}

\subsection{$\pi_1$-small subvarieties}

The following type of subvarities occur naturally in the study of the Shafarevich conjecture and the construction of Shafarevich maps (aka $\Gamma$-reduction)(\cite{Camp94}, \cite{Kol93}, \cite{Kol95}). We follow \cite[Definition 2]{DOKR02} to make the following definition.
\begin{definition}\label{def:pi1small}
Let $X$ be a complex normal variety. A subvariety $Y\subset X$ is called \textsl{\textsl{$\pi_1$-small}} if for some (hence for any) resolution of singularities $\widehat{Y}\to Y$, the homomorphism $ \pi_1(\widehat{Y})\to \pi_1(X)$ of fundamental groups has finite image.  We call $Y$  \textsl{$\pi_1$-trivial} if $\pi_1(\widehat{Y})\to \pi_1(X)$ is trivial. \ A Weil divisor $D=\sum_in_iD_i$ is called a \textsl{$\pi_1$-small ($\pi_1$-trivial) divisor} if each component $D_i$ is $\pi_1$-small ($\pi_1$-trivial).
\end{definition}

Any rational subvariety is $\pi_1$-small. The smooth model of $\pi_1$-small subvariety might have large fundamental group. For example, one can take a high genus curve in the exceptional divisor of a blowup of a higher dimensional smooth variety. 
$\pi_1$-small subvarieties play a central role in the following theorem on the existence of Shafarevich maps (aka $\Gamma$-reductions) due to Campana \cite{Camp94} and Koll\'ar \cite{Kol93}, \cite{Kol95} individually.

\begin{theorem}[Campana-Koll\'ar]\label{Thm:Shafarevich}
Let $X$ be a normal variety, and $H$ be a normal subgroup of $\pi_1(X)$. Then there is a dominant rational map $\Sh^H_X\colon X\dashrightarrow\Sh^H(X)$ to a normal variety $\Sh^H(X)$ such that 
\begin{enumerate}
\item $\Sh^H_X$ has connected fibres, and

\item there are countably many closed proper subvarieties $D_i\subset X, i\in I$ such that for every irreducible subvariety $Z\subset X$ and $Z\not\subset \bigcup_i D_i$, $Z$ is contracted by $\Sh^H_X$ if and only if the image of the natural homomorphism $\pi_1(\widehat{Z})\to \pi_1(X)/H$ has finite image, where $\widehat{Z}$ is a desingularization of $Z$.
\end{enumerate}
\end{theorem}

The map $\Sh^H_X$ and the normal variety $\Sh^H(X)$ are called the \textsl{Shafarevich map} and the \textsl{Shafarevich variety} of $X$ with respect to $H$, respectively. The Shafarevich map and the Shafarevich variety are unique up to birational equivalence. If $H=\{0\}$, then we simply denote $\Sh_X\coloneqq\Sh_X^{\{0\}}$ and $\Sh(X)\coloneqq\Sh^{\{0\}}(X)$. Notice that $\pi_1(X)$ is finite if and only if $\dim \Sh(X)=0$. If $H$ is the kernel of a representation $\rho\colon \pi_1(X)\to \GL(m, \C)$, then we write $\Sh^{\rho}_X\coloneqq\Sh_X^{H}$ and $\Sh^{\rho}(X)\coloneqq\Sh^{H}(X)$.

\begin{definition}\label{def:glfg}
A normal variety $X$  has \textsl{generically large fundamental group}, if $X$ satisfies the following two equivalent conditions
\begin{enumerate}
\item For any very general point $x\in X$ and any positive dimensional subvariety $Y$ passing through $x$, the image $\pi_1(\widehat{Y})\to\pi_1(X)$ is infinite, where $\widehat{Y}$ is a desingularization of $Y$.
\item The Shafarevich map $\Sh_X$ of $X$ is birational.  
\end{enumerate}
\end{definition}

\subsection{Representations of fundamental groups}

In this subsection, we recall several terminologies on representations of finitely presented groups.

Let $G$ be a finitely presented group with a set of generators $\{g_1,\ldots, g_r\}$. A linear representation $\rho\colon G\to \GL(m, \C)$ can be viewed as a closed point of the representation variety $\Hom(G, \GL(m, \C))\subset\GL(m ,\C)^r$.
\begin{definition}\label{df:bettimoduli}
The $m$-th \textsl{Betti moduli space} $\MB(G, m)$ of $G$ is defined to be the GIT quotient $\Hom(G, \GL(m, \C))/\!\!/ \GL(m, \C)$ by the simultaneous conjugate action of $\GL(m, \C)$. When $G\coloneqq\pi_1(X)$ is the fundamental group of an algebraic  variety $X$, we denote $\MB(G, m)$ by $\MB(X,m)$.
\end{definition}

Given a representation $\rho\colon G\to\GL(m, \C)$, $\rho$ is called \textsl{rigid} if the corresponding closed point $[\rho]\in\MB(G, m)$ is isolated. We call $\rho$ \textsl{semisimple} if it is a direct sum of irreducible representations, and we call $\rho$ \textsl{unitary} if the image $\rho(G)$ is contained in the unitary group $U(m)$. The closed points of the Betti moduli space $\MB(G, m)$ one-to-one correspond to $m$ dimensional complex semisimple representations of $G$.

\subsection{Parabolic Higgs bundles and complex variation of Hodge structures}\label{sect:PHBCVHS}

In this section, we give a brief discussion on parabolic Higgs bundles and the complex variation of Hodge structures, and recall some results that will be used in the proof of our main theorem. We first recall the definition of parabolic Higgs bundles with the notation in \cite[Section 3]{DH19}. Let $X$ be a complex smooth projective variety with $D=\sum_{i=1}^sD_i$ a simple normal crossing divisor. Denote $U\coloneqq X-D$ and the inclusion $j\colon U\hookrightarrow X$.

\begin{definition}\label{df:para-Higgs}
A \textsl{parabolic Higgs bundle} $(E, \ba E, \theta)$ on $(X, D)$ is a locally free $\mathcal{O}_U$-module $E$, together with an $\mathbb{R}^s$-indexed filtration $\ba E$ (\textsl{parabolic structure}) by locally free subsheaves such that
\begin{enumerate}
\item \label{def:parab1} $\bfa\in\mathbb{R}^s$ and $\ba E|_U=E$;

\item \label{def:parab2} For $1\leq i\leq s$, ${}_{\bm{a}+\bm{1}_i}E=\ba E(-D_i)$, where $\bm{1}_i=(0,\ldots,0,1,0,\ldots,0)$ with $1$ in the $i$-th component;

\item \label{def:parab3} ${}_{\bm{a}-\bm{\epsilon}}E=\ba E$ for any vector $\bm{\epsilon}=(\epsilon,\ldots, \epsilon)\in \mathbb{R}^s$ with $0<\epsilon\ll 1$;

\item \label{def:parab4} The set of \textsl{jump weight vectors} $\{\bm{a}\mid\ba E/{}_{\bm{a}+\bm{\epsilon}}E \neq0$ for any vector $\bm{\epsilon}=(\epsilon,\ldots, \epsilon)\in \mathbb{R}^s$ with $0<\epsilon\ll 1\}$ is discrete in $\mathbb{R}^s$,
\end{enumerate}
and an $\mathcal{O}_X$-linear map (\textsl{Higgs field}) \[\theta\colon {}_{\bm{0}}E\to \Omega_X^1(\log D)\otimes {}_{\bm{0}}E\] such that
\begin{enumerate}
\item \label{def:parab5} $\theta\wedge\theta=0$, and

\item \label{def:parab6} $\theta(\ba E)\subseteq\Omega_X^1(\log D)\otimes\ba E$ for $\bm{a}\in [0,1)^s$.
\end{enumerate}
\end{definition}

We will denote ${}_{\bm{0}}E$ by $\diae$. %A weight is normalized if it lies in $[0,1)^s$. 
Note that the parabolic structure is uniquely determined by the filtration for weights in $[0,1)^s$. We call $(E, \ba E, \theta)$ a \textsl{parabolic Higgs sheaf}, if we only assume that $E$ is torsion free and $\ba E$ are coherent subsheaves in the above definition (This terminology will be used in the stability condition in Section \ref{sect:st-chern}).

One important class of parabolic Higgs bundle on $(X, D)$ comes from variation of Hodge structures over $U$. We recall the definition of a  complex polarized variation of Hodge
structures ($\C$-PVHS for short) from the $\scrC^\infty$ point of view \cite{Simp92}. 

\begin{definition}\label{df:CVHS}
A complex polarized variation of Hodge structures of weight $w$ over $U=X-D$ is a triple $(V, \mathcal{D}, k_V)$ consisting of a $\scrC^{\infty}$-vector bundle $V$ with a decomposition into $\scrC^{\infty}$-subbundles $$V=\oplus_{i+j=w}V^{i,j};$$ a flat connection $\mathcal{D}$ satisfying Griffiths's transversality $$\mathcal{D}\colon V^{i,j} \longrightarrow A^{0,1}(V^{i+1, j-1})\oplus
A^{1,0}(V^{i,j})\oplus A^{0,1}(V^{i, j}) \oplus A^{1,0}(V^{i-1, j+1})$$ with $A^{p,q}(V^{i,j})$ being the space of $\scrC^{\infty}$ $(p,q)$-forms valued in $V^{i,j}$; and a flat hermitian form $k_V$ which is positive (negative) definite on $V^{i,j}$
when $i$ is even (odd) such that the decomposition $V=\oplus_{i+j=w}V^{i,j}$ is orthogonal with respect
to $k_V$. 
\end{definition}

\begin{remark}\label{rmk:vhs-higgs-const} Decompose
$\mathcal{D}$ into operators of types $(1,0)$ and $(0,1)$ 
$$
\mathcal{D}=\nabla+\bar{\partial} 
.$$
Then $\bar{\partial}$ defines a complex structure on $V$. Let $\mathcal{V}$ denote the corresponding holomorphic bundle. Then we have holomorphic subbundles $F^p\mathcal{V}=\oplus_{i\geq p}V^{i,j}$ and a holomorphic flat connection $\nabla$ such that $\nabla(F^p\mathcal{V})\subset
\Omega_U^1\otimes F^{p-1}\mathcal{V}$. The graded holomorphic bundle 
 $E\coloneqq\Gr_F\mathcal{V}$ carries holomorphic subbundles $E^{i, w-i}=\Gr^i_F\mathcal{V}/\Gr^{i+1}_F\mathcal{V}$ and a Higgs field $\theta\coloneqq\Gr_F\nabla\colon E^{i, w-i}\to \Omega_U^1\otimes E^{i-1, w-i+1}$ which is $\mathcal{O}_U$-linear. Then $(E, \theta)$ is a Higgs bundle on $U$. To explain how to extend the Higgs bundle $(E, \theta)$ to a parabolic Higgs bundle over $(X, D)$, we recall a more general way to construct Higgs bundle from a flat bundle with harmonic metric in \cite{Simp92}.\end{remark}
 
Given a $\scrC^{\infty}$ flat bundle $(V, \mathcal{D})$ with Hermitian metric $K$, we can decompose
$\mathcal{D}=\mathcal{D}^{1,0}+\mathcal{D}^{0,1}$ into operators of types $(1,0)$ and $(0,1)$. Let $\delta'$ and $\delta''$ be operators of type $(1, 0)$ and $(0, 1)$ such
that \[\langle\mathcal{D}^{0,1}u, v\rangle_K+\langle u, \delta'v\rangle_K=\mathcal{D}^{0,1}\langle u, v \rangle_K;\]
\[\langle\delta''u, v\rangle_K+\langle u, \mathcal{D}^{1,0}v\rangle_K=\delta''\langle u, v \rangle_K,\]
for all local sections $u, v$ of $V$. Define \[\bar{\partial}\coloneqq\frac{1}{2}(\mathcal{D}^{0,1}+\delta'');\ \  \theta\coloneqq\frac{1}{2}(\mathcal{D}^{1,0}-\delta').\]
We call the triple $(V, \mathcal{D}, K)$ a \textsl{harmonic bundle} if $\Lambda(\bar{\partial}+\theta)^2=0$, where $\Lambda$ is the usual adjoint operator of wedge of K\"ahler form coming from a fixed polarization of $X$. In this case,  we get a Higgs bundle $(E, \theta)$ (see e.g., \cite[Lemma 1.1]{Simp92}), where $E$ is the holomorphic vector bundle given by $V$ together with the new complex structure $\bar{\partial}$. 
$(V, \mathcal{D}, K)$ is called a \textsl{tame harmonic bundle} if it is a harmonic bundle and the eigenvalues
of the associated Higgs field $\theta$ (which are multivalued 1-forms) have poles of order
at most 1 around the boundary divisor $D$. 
For the flat bundle $(V, \mathcal{D}, k_V)$ coming from a $\C$-PVHS above, by changing the sign of $k_V|_{V^{i,j}}$ for odd $i$, we get a flat positive definite Hermitian form $K_V$. Then $(V, \mathcal{D}, K_V)$ is a tame harmonic bundle(See e.g., \cite[Proposition 5.4]{AHL19}). %Firstly, Deligne \cite[Proposition 1.13]{Del87} shows that the local system/flat bundle associated to $\C$-PVHS is semisimple. Secondly, a celebrated result
%given by J. Jost and K. Zuo [29], which says there 
%\begin{theorem}[Deligne]
%Let $U$ be a quasiprojective variety, and $(V, \mathcal{D}, k_V)$ be a $\C$-PVHS over $U$. Then the local system $\ker\mathcal{D}$ is semisimple. 
%\end{theorem}

\begin{proposition}\label{thm:vhs-tame}
Let $(V, \mathcal{D}, k_V)$ be a complex polarized variation of Hodge structures, then the corresponding flat bundle $(V, \mathcal{D}, K_V)$ is a tame harmonic bundle.
\end{proposition}

For a tame harmonic bundle $(V, \mathcal{D}, K)$, the associated Higgs bundle $(E, \theta)$ over $U$ can be extended to a parabolic Higgs bundle $(E, \ba E, \theta)$ over $X$. Around any point $p\in D$, we take an analytic neighbourhood $X_p\coloneqq\Delta^n$ in $X$ with local coordinate functions $\{z_i\}_{i=1}^n$ such that $D|_{X_p}\coloneqq\bigcup_{i=1}^r\{z_i=0\}$. %Then the above resulting Higgs field can be written as \[\theta=\sum_{i=1}^rf_i\frac{dz_i}{z_i}+\sum_{j=r+1}^ndz_j.\]
Then 
$\ba E\subset j_*E$ is generated by local sections \[\Gamma(X_p, \ba E)\coloneqq \{v\in \Gamma(X_p\cap U, E)|\ |v|_{K}\leq C\cdot\Pi_{i=1}^s|z_i|^{-a_i-\epsilon}, \] \[\text{for some real number}\ C>0\ \text{and any}\ \epsilon>0.\}\]

\subsection{Stability and Chern classes of parabolic Higgs bundles}\label{sect:st-chern}

Parabolic Higgs bundles induced from tame harmonic bundles possess more properties on stability condition and Chern classes. In general, for a parabolic Higgs bundle
$(E, \ba E, \theta)$ on $(X,D)$, its \textsl{parabolic Chern classes} $\pc_i(E)$ are defined to be the usual Chern
classes $c_i(\diae)$ of $\diae\coloneqq{}_{\bm{0}}E$ together with modifications along the boundary divisor $D$ (see e.g., \cite[Section 3]{AHL19} and \cite[Section 3.1]{Moc06} for
more details). With a fixed ample line bundle $H$ on $X$, its \textsl{parabolic degree}
$\pd(E)$ is defined to be the intersection number $\pc_1(E)\cdot H^{n-1}$. We call $(E, \ba E, \theta)$ \textsl{$\mu_H$-stable} (\textsl{$\mu_H$-semistable}) if
for any coherent torsion free subsheaf $V$ of $\diae$, with $0< \rank V<\rank\diae$ and $\theta(V)\subseteq \Omega_X^1(\log D)\otimes V$, the condition
\[\frac{\pd(V)}{\rank(V)}<\frac{\pd(E)}{\rank(E)} (\frac{\pd(V)}{\rank(V)}\leq\frac{\pd(E)}{\rank(E)})\]
is satisfied (Here $V$ carries the induced parabolic structure from $(E, \ba E, \theta)$, i.e., $\ba V\coloneqq V\cap \ba E$. Hence one can still define the parabolic degree of this induced parabolic Higgs sheaf). A parabolic Higgs bundle $(E, \ba E, \theta)$ is called \textsl{$\mu_H$-polystable} if it is a direct sum of $\mu_H$-stable
parabolic Higgs bundles. A parabolic Higgs bundle $(E, \ba E, \theta)$ is called \textsl{locally abelian} if Zariski locally $(E, \ba E)$ is a direct sum of filtered line bundles (see e.g.,\cite{IS07}). Then we are ready to state the following correspondence by Simpson \cite[Main Theorem]{Simp90} for noncompact
curves and Mochizuki \cite[Theorem 1.4]{Moc06} for higher dimensional quasiprojective varieties.
\begin{theorem}[Kobayashi-Hitchin Correspondence]\label{Thm:K-H-corresp}
Let $X$ be a smooth complex projective variety with a simple normal crossing divisor $D$ and an ample line bundle $H$. Then there is a one-to-one correspondence between tame harmonic bundles $(V, \mathcal{D}, K_V)$ and $\mu_H$-polystable locally abelian parabolic Higgs bundles $(E, \ba E, \theta)$ with trivial parabolic Chern classes.
\end{theorem}
The correspondence coincides with the above construction from a tame harmonic bundle to a parabolic Higgs bundle. To end this subsection, we recall the following perturbation trick in \cite{AHL19}. For a parabolic Higgs bundle $(E, \ba E, \theta)$, one can fix the filtered locally free subsheaves in the parabolic structure, i.e., the set of locally free sheaves $\{\ba E| \bm{a}\in [0,1)^s\}$, and perturb each jump weight vector in $\{\bm{a}\mid\ba E/{}_{\bm{a}+\bm{\epsilon}}E \neq0$ for any vector $\bm{\epsilon}=(\epsilon,\ldots, \epsilon)\in \mathbb{R}^s$ with $0<\epsilon\ll 1\}$. More specifically, for each jump weight vector $\bm{a}=(a_1, \ldots, a_s)\in [0,1)^s$, one can perturb $\bm{a}$ to be $(a_1+\epsilon_1, \ldots, a_s+\epsilon_s)$ with $0\leq\epsilon_i\ll1$ and redefine the jump weight vector of the locally free subsheaf $\ba E$ to be $(a_1+\epsilon_1, \ldots, a_s+\epsilon_s)$. After we perturb all jump weight vectors in $[0,1)^s$, we get a new parabolic Higgs bundle $(E, \baa E, \theta)$. We say that $(E, \ba E, \theta)$ is \textsl{$\epsilon$-close} to $(E, \baa E, \theta)$, if for each locally free subsheaf $F$ in the parabolic structure of  $(E, \ba E, \theta)$ and $(E, \baa E, \theta)$, we have $|a_i-a'_i|\leq\epsilon$ for each $1\leq i\leq s$, where $(a_1, \ldots, a_s)$ and $(a'_1, \ldots, a'_s)$ are the jump weight vectors of $(E, \ba E, \theta)$ and $(E, \baa E, \theta)$ corresponding to $F$, respectively.
\begin{proposition}[{\cite[Lemma 3.3, Lemma 6.1]{AHL19}}]\label{lem:Perturb}
Let $(E, \ba E, \theta)$ be a locally abelian $\mu_H$-polystable parabolic Higgs bundle with trivial
parabolic Chern classes on $(X, D)$. For any $0\leq\epsilon\ll1$, there exists a locally abelian $\mu_H$-polystable parabolic Higgs bundle  $(E, \baa E, \theta)$ with trivial parabolic Chern classes such that $(E, \baa E, \theta)$ is $\epsilon$-close to $(E, \ba E, \theta)$ and all jump weight vectors of $(E, \baa E, \theta)$ are in $\Q^s$.
\end{proposition}

\begin{remark}
The polystability of $(E, \baa E, \theta)$ follows from the proof of \cite[Lemma 6.1]{AHL19}. In \cite{AHL19}, the parabolic Higgs bundles are $\R$-indexed. Under the ``locally abelian'' assumption, they are equivalent to $\R^s$-indexed parabolic Higgs bundles in Definition \ref{df:para-Higgs} (\cite[Page 3]{AHL19}).
\end{remark}

\subsection{Semi-negativity}\label{sect:Semineg}

As is described in Section \ref{sect:PHBCVHS}, one can construct a Higgs bundle $(E=\oplus_iE^{i, w-i}, \theta)$ over $U$ for a given $\C$-PVHS $(V, \mathcal{D}, k_V)$ of weight $w$. Moreover, if we assume the local monodromy of $(V, \mathcal{D}, k_V)$ around each component of $D$ is unipotent, then the extended parabolic Higgs bundle $(E, \ba E, \theta)$ has the trivial parabolic structure, i.e., $\diae={}_{\bm{0}}E$ is the canonical Deligne extension and the only weight vector in $[0,1)^s$ is $\bm{0}$. In this case, the holomorphic subbundles $F^i\mathcal{V}$ and hence $E^{i, w-i}$ are canonically extended across $D$ by Schmid \cite{Sch73}, and $\theta$ has at most logarithmic poles along $D$. Denote the extended Higgs bundle by $(\diae=\oplus_i \diae^{i, w-i}, \theta)$ with $\theta\colon \diae^{i, w-i}\to \Omega_X^1(\log D)\otimes \diae^{i-1, w-i+1}$.  Recall the following semi-negativity theorem for the above Higgs bundle due to Zuo \cite[Theorem 1.2]{Zuo00}, which is a key ingredient in the proof of our main theorem.
\begin{theorem}[Zuo]\label{Thm:semi-neg}
Let $(\diae=\oplus_i \diae^{i, w-i}, \theta)$ be the extended Higgs bundle on $(X, D)$ coming from a $\C$-PVHS over $U$ with unipotent local monodromy around $D$. Suppose $N\subset \diae^{i, w-i}$ is a subbundle with $\theta(N)=0$.  Then for any morphism $f\colon Z\to X$, \[\int_Zf^{\ast}(c_1^{i_1}(N^{\vee})\ldots c_n^{i_n}(N^{\vee}))\geq 0,\] where $N^{\vee}$ is the dual bundle of $N$. In particular, \[\int_Zf^{\ast}\overset{\dim Z}{\wedge}(-c_1(N))\geq 0.\]
\end{theorem}

\subsection{Logarithmic 1-forms and spectral coverings of unbounded $p$-adic representation}\label{sect:p-adicrep}

In this subsection we recall logarithmic 1-forms and spectral coverings associated to unbounded $p$-adic representations of fundamental groups of quasiprojective varieties from the celebrated work due to Jost and Zuo \cite{JZ97} (see also \cite[Section 4.1.4]{Zuo99} and \cite[Part 5]{Moc07b}). 

For a nonrigid semisimple representation $\rho\colon \pi_1(U)\to \SL(m, \C)$, one can get an unbounded $p$-adic representation (due to Simpson \cite{Simp92}). More specifically, there is an affine curve $A\subset \Hom(\pi_1(U), \SL(m, \C))$ passing through $\rho$ such that $A$ maps to an affine curve in $\MB(X, m)$% and get a reductive representation $\rho_T:\pi_1(U)\to \SL(n, k(T))$ coming from the representation for each $t\in T$
. One can complete the affine curve $A$ with points at infinity and desingularize it along points at infinity. Denote the resulting complete curve by $\overline{A}$. Choosing a smooth point $\infty$ at infinity, one has the $\infty$-adic valuation $\nu_{\infty}(\bullet)$ on function field $k(A)$ of $A$, where $\nu_{\infty}(f)$ is the vanishing order of $f$ at $\infty$. Then we get the completion  $k(A)_{\infty}$ of $k(A)$ with respect to the the $\infty$-adic norm $|f|\coloneqq c^{-\nu_{\infty}(f)}$ ($c$ is a fixed real number and $c>1$). The representation $\rho$ induces a new representation $$\rho_{A}\colon \pi_1(U)\to \SL(m, k(A)_{\infty}),$$ with $\rho_A|_{\rho}=\rho$. Also, notice that  $\rho_A$ are unbounded in the sense that the entries of matrices of $\im \rho_A$ is unbounded in $\SL(m, k(A)_{\infty})$ with respect to the above $\infty$-adic norm (see e.g., \cite[Lemma 2.2.3]{Zuo99}). With the $\infty$-adic unbounded representation $\rho_A$, by \cite[Theorem 1.1]{JZ97} we have a $\rho_A$-equivariant nonconstant pluriharmonic map $u_{\rho}\colon \widetilde{U}\to \Delta(\SL(m, k(A)_{\infty})$ from the universal cover of $U$ to the Bruhat-Tits building $\Delta(\SL(m, k(A)_{\infty})$ (refer to \cite{BT72} and \cite{GS92} for the construction of $\Delta(\SL(m, k(A)_{\infty})$ and $u_{\rho}$ for projective case). %which is the $p$-adic analog of $\GL(n, \C)/U(n, \C)$. For a reductive representation $\rho: X\to \GL(n, \C)$, one can find a harmonic bundle $V$, hence a map $\tau$ from universal cover of $X$ to $\GL(n, \C)/U(n, \C)$. The above map $u$ can be seen as the $p$-adic analog of $\tau$) by \cite[Theorem 3.1.2]{Zuo99}{\color{red} Due to Gromov and Schoen}.

The complexified differential of $u_{\rho}$ gives a nontrivial $\pi_1(U)$ multi-valued holomorphic 1-form on $\tilde{U}$, which descends to a nontrivial multi-valued holomorphic 1-form $\omega$ on $U$. By the controlled estimate of $du_{\rho}$ at the boundary divisor $D$, $\omega$ can be extended to a multi-valued logarithmic 1-form on $X$ with at most logarithmic poles along $D$. %(see e.g., \cite[Lemma 4.1.4]{Zuo99}). 
We still denote this multi-valued logarithmic 1-form on $X$ by $\omega$. Then there is a finite ramified Galois covering (\textsl{spectral covering}) $\pi\colon X^s\to X$ from a normal projective variety $X^s$ (not necessarily smooth) to $X$ such that the pullback $\pi^*\omega$ splits into nontrivial  1-forms $\omega_1, \dots, \omega_l$ in $H^0(X^s, \pi^*\Omega_{X}^1(\log D_{\infty}))$ which have log poles at infinity. Also, the ramified loci of $\pi$ is contained in the union of the zero loci of 1-forms $\{\omega_i-\omega_j\}_{1\leq i<j\leq l}$ (see e.g., \cite[Section 3]{JZ97} and \cite[Section 4.1.4]{Zuo99} for more details). 

On the other hand, if $\rho\colon \pi_1(U)\to \SL(m, \C)$ is a rigid semisimple representation, $\rho$ is valued in some number field $K$ since $M_{\text{B}}(U, m)$ is defined over $\Z$. For the ring of integers $\mathcal{O}_K$ in $K$, assume that $[\rho(\pi_1(U))\colon \rho(\pi_1(U))\cap \SL(m, \mathcal{O}_K)]=\infty$, one can find a prime ideal $\mathfrak{p}\in \mathcal{O}_K$ such that the induced representation $\rho_{\mathfrak{p}}\colon \pi_1(U)\to \SL(m, K_{\mathfrak{p}})$ is a $\mathfrak{p}$-adic unbounded representation with respect to the $\mathfrak{p}$-adic norm on the number field $K_{\mathfrak{p}}$ (see e.g., \cite[Page 122]{Zuo99}). Then by \cite[Theorem 1.1]{JZ97} we have a $\rho_{\mathfrak{p}}$-equivariant nonconstant pluriharmonic map $u_{\rho}\colon \widetilde{U}\to \Delta(\SL(m, K_{\mathfrak{p}})$. Similarly, we get spectral covers $X^s$ of $X$ and logarithmic 1-forms on $X^s$ as before.

\section{$\pi_1$-small Divisors on Quasi-projective Varieties}

In this section we will study $\pi_1$-small divisors on quasi-projective varieties and prove Theorem \ref{Thm:main}. We will essentially follow the strategy of the proof of Theorem \ref{Thm:L-R} (see \cite{LR96} and \cite{Zuo99}) for projective surfaces, and use deformation and perturbation methods (\cite[Section 10.1]{Moc06} and \cite[Section 2]{AHL19}), and the covering trick due to Kawamata to deal with issues raised from the boundary.  

We fix the following notations throughout this section. Let $X$ be a complex smooth projective variety of dimension $n\geq2$, $D=\sum_{k=1}^sD_k$ be a simple normal crossing divisor, and  $Y=\sum_{i=1}^r n_iY_i$ be a divisor on $X$ such that $D$ and $Y$ share no common component. Denote the quasiprojective variety $X-D$ by $U$, and the singular set of $D$ by $\Sing D$. We also denote the restriction of $Y_i$ on $U$ by $Y^o_i$ for each $i$, and $Y^o\coloneqq\sum_{i=1}^r n_iY_i^o$. For a nontrivial logarithmic 1-form $\eta\in H^0(X, \Omega^1_X(\log D))$, we say that \textsl{$\eta$ restricts to 0 on $Y$} and write that $\eta|_Y=0$, if the restricted logarithmic 1-form of $\eta$ to the smooth locus of $Y_i$ is zero for all $i$.

\subsection{Logarithmic 1-forms and quasi-Albanese maps} 

Firstly, we recall the construction of quasi-Albanese map of $U= X-D$  from Iitaka \cite{Ii76}.
Pick holomorphic 1-forms $\{\omega_1,\ldots\omega_q\}$ in $H^0(X, \Omega_X^1)$ and logarithmic 1-forms $\{\eta_1, \ldots, \eta_m\}$ in $H^0(X, \Omega_X^1(\log D))$ such that $\{\omega_1, \ldots, \omega_q, \eta_1, \ldots, \eta_m\}$ form a basis of $H^0(X, \Omega_X^1(\log D))$. Dually one can pick a basis $\{\alpha_1, \ldots, \alpha_{2q}\}$ for $H_1(X, \mathbb{Z})$ and a basis $\{\beta_1, \ldots, \beta_m\}$ for $ \ker\{ H_1(U, \mathbb{Z})\to H_1(X, \mathbb{Z})\}.$ With respect to the dual basis $\omega_i^*, \eta_j^*$ in $H^0(X, \Omega_X^1(\log D))^{\vee}$, we have the following periods
\[\Lambda\coloneqq\sum_{i=1}^{2q}\mathbb{Z}\left(\int_{\alpha_i}\omega_1, \ldots, \int_{\alpha_i}\eta_m\right)+\sum_{j=1}^m\mathbb{Z}\left(\int_{\beta_j}\omega_1, \ldots, \int_{\beta_j}\eta_m\right).\]
The quasi-Albanese variety is defined to be $A_U =\displaystyle \frac{H^0(X,\Omega_X^1(\log D))^{\vee}}{\Lambda}$ and the quasi-Albanese map $\alpha_U\colon U \to A_U$ is given by 
\[\alpha_U(x)=\left[\sum_{i=1}^{q}\left(\int_{p}^x\omega_i\right)\omega^*_i+\sum_{j=1}^m\left(\int_{p}^x\eta_j\right)\eta_j^*\right]\slash\Lambda,\]
 where $p\in U$ is a chosen base-point in $U$.

\begin{lemma}\label{lem:quasiprospurr}
Let $X$ be a smooth projective surface with a simple normal crossing divisor $D=\sum_{i=1}^sD_i$ and a divisor $Y=\sum_{i=1}^rn_iY_i$ such that $Y$ intersects mildly with respect to $D$. %, i.e., $Y\cap \Sing D=\emptyset$. 
Suppose that there is a nontrivial logarithmic 1-form $\eta\in H^0(X, \Omega_X^1(\log D))$ such that $\eta|_Y=0$, then $Y^2\leq 0$. 
\end{lemma}

\begin{proof}
Without loss of generality, we can assume the support $\bigcup_{i=1}^rY_i$ of $Y$ is connected. Consider the quasi-Albanese morphism $\alpha_U\colon U\to A_U$, where $A_U$ is an extension \[\xymatrix{
1\ar[r] & {\C^*}^{m}\ar[r] &A_U\ar[r]^-{\Phi} &A_X\ar[r] &1},\] and $\alpha_X\colon X\to A_X$ is the Albanese map of $X$. Let $B\coloneqq \langle \alpha_U(\bigcup_{i=1}^rY_i^o)\rangle$ be the subtori generated by $\alpha_U(\bigcup_{i=1}^rY_i^o)$ in $A_U$, i.e., the intersection of semi-abelian subvarieties such that some translates of them contain $\alpha_U(\bigcup_{i=1}^rY_i^o)$ (Here $Y^o_i\coloneqq Y_i\cap U$). Then we get the morphism $\alpha\colon U\to A_U/B$ induced by $\alpha_U$, with $A_U/B$ being an extension \[\xymatrix{
1\ar[r] & {\C^*}^{l}\ar[r] &A_U/B\ar[r]^-{\Psi} &A\ar[r] &1,}\] where $A$ is the quotient of $A_X$ by $\Phi(B)$. Denote the quotient map by $p\colon A_X\to A$.  Since there is a nontrivial logarithmic 1-form $\eta$ such that $\eta|_Y=0$, $A_U/B$ is nontrivial. We denote the space of logarithmic 1-forms on $A_U/B$ by $L(A_U/B)$. Notice that $A_U/B$ admits a natural compactification $\overline{A_U/B}$, such that there is a projection $\pi\colon \overline{A_U/B}\to A$ extending $\Psi$ and making $\overline{A_U/B}$ a $(\mathbb{P}^1)^l$-bundle over $A$ (see e.g., \cite[Section 2]{NW02} for the construction). Then the morphism $\alpha$ extends to a rational map $\overline{\alpha}\colon X\dashrightarrow \overline{A_U/B}$. 

Claim that $\overline{\alpha}$ is regular outside of $\Sing D$. We follow the idea in \cite[Lemma 2.4]{NW02} to prove the claim. For each point $x_0\in D\backslash \Sing D$, there is an analytic open neighborhood $W$ of $p\circ \alpha_X(x_0)\in A$ such that $\pi^{-1}(W)=(\mathbb{P}^1)^l\times W$. Taking an analytic local neighborhood $\Delta_{x_0}$ of $x_0$ in $X$, we can write \[\alpha(x)=(f_1(x), \ldots, f_l(x), g(x)),\] on $\Delta_{x_0}^o\coloneqq\Delta_{x_0}\cap U$, where $g$ is a holomorphic map from $\Delta_{x_0}$ to $W$, and $$f_i(x)\coloneqq \exp(\int_a^x\theta_i)$$ with a base point $a\in\Delta_{x_0}^o$ and a set of logarithmic 1-forms $\{\theta_1, \ldots, \theta_l\}$  in $L(A_U/B)\subset H^0(X, \Omega_X^1(\log D))$ such that their images in the quotient space $L(A_U/B)/H^0(A, \Omega_A^1)$ form a basis (Recall that $L(A_U/B)$ denotes the space of logarithmic 1-forms on $A_U/B$). Denote a small cycle around the boundary component $D_k$ containing $x_0$ by $\delta_k$. Then one can extend $\alpha$ across $x_0$ by defining that
 \[
    f_i(x_0)= \left\{\begin{array}{lr}
        0, & \text{if } \int_{\delta_k}\theta_i>0;\\
       \infty, & \text{if } \int_{\delta_k}\theta_i<0;\\
       \exp(\int_a^{x_0}\theta_i)\in\C^*, & \text{if } \int_{\delta_k}\theta_i=0.
        \end{array}\right.
  \]

Notice that since $Y$ intersects mildly with respect to $D$, $Y\cap \Sing D=\emptyset$. Then $\overline{\alpha}$ is regular along $Y$ and contracts $Y$. After blowing up $X$ without affecting $Y$, we have a morphism $\varphi\colon X'\to \overline{A_U/B}$ which contracts $Y$. When $\dim \varphi(X')=1$, we get $Y^2\leq 0$ since the intersection matrix of components of any fibre of $\varphi$ is negative semi-definite. If $\dim \varphi(X')=2$, we get $Y^2<0$ by \cite[Theorem 10.1]{KK13}. Hence $Y^2\leq 0$ on $X$.
\end{proof}

\begin{corollary}\label{cor:trivialh1}
Let $X$, $D$, $Y$, and $U\coloneqq X-D$ be the same notations as in Lemma \ref{lem:quasiprospurr}. Suppose that 
\begin{enumerate}
\item the composition \textup{$\pi_1(\widehat{Y}^o_i)\to \pi_1(U)\to \pe(U)$} is trivial for all $i$, where $\widehat{Y}^o_i$ is any desingularization of $Y^o_i$, and \textup{$\pe(U)$} is the \'etale fundamental group of $U$;
\item $Y^2>0$. 
\end{enumerate}
Then $H^1(U, \C)=0$.
\end{corollary}

\begin{proof}
Assume that $H^1(U, \C)\neq0$, then there exists a nontrivial logarithmic 1-form $\eta\in H^0(X, \Omega_X^1(\log D))$ such that $\eta|_Y=0$, since $\pi_1(\widehat{Y}^o_i)\to \pi_1(U)\to\pe(U)$ is trivial for all $i$.  However, this contradicts to Lemma \ref{lem:quasiprospurr}.
\end{proof}

\subsection{Finiteness of Betti moduli}

In this subsection we show the following finiteness property of the Betti moduli space, which is a generalization of Zuo \cite[Lemma 5.4.6]{Zuo99}.

\begin{proposition}\label{prop:finitebetti}
Let $X$ be a smooth projective variety with a normal crossing divisor $D=\sum_{i=1}^sD_i$ and a divisor $Y=\sum_{i=1}^rn_iY_i$. For $U\coloneqq X-D$ and $Y^o\coloneqq Y|_U$,  suppose that 
\begin{enumerate}
\item \label{prop-finite:1}the composition \textup{$\pi_1(\widehat{Y}^o_i)\to \pi_1(U)\to \pe(U)$} is trivial for all $i$, where $\widehat{Y}^o_i$ is a desingularization of $Y^o_i$, and \textup{$\pe(U)$} is the \'etale fundamental group of $U$,

\item \label{prop-finite:2} $Y$ intersects mildly with respect to $D$, and

\item \label{prop-finite:3} there exists an ample divisor $H$ on $X$, such that $Y^2\cdot H^{n-2}>0$. 
\end{enumerate}
Then \textup{$\dim \MB(U, m)=0$} for each $m$. Moreover, any semisimple representation of each conjugacy class in \textup{$\MB(U, m)$} is valued in some algebraic integers $\mathcal{O}_K$ after passing to a finite \'etale cover of $U$.
\end{proposition}

\begin{proof} Without loss of generality, we can assume that $Y$ is effective and $H$ is very ample. Take a smooth complete intersection surface $T$ of $n-2$ general hyperplanes of the linear system $|H|$ such that $D|_T$ is a simple normal crossing divisor on $T$.
% such that scheme theoretic intersection $Y_i\cdot H_1\cdot\ldots\cdot H_{n-2}$ are integral curves $C_i$ for all $i$
We denote

\begin{itemize}
\item $B_i\coloneqq D_i|_T$, $B\coloneqq \sum_{i=1}^sB_i$;
\item $V\coloneqq T-B$;
\item $C_i^o\coloneqq Y_i^o|_T=Y_i|_V$, which is irreducible by Bertini theorem;
\item $C\coloneqq Y|_T=\sum_{i=1}^{r} n_i C_i$, where $C_i\coloneqq {Y_i}|_T$;
\item $\widehat{C}^o_i$ are desingularizations of $C_i^o$.
\end{itemize}

By the Lefschetz hyperplane theorem for quasiprojective varieties (see \cite[Theorem 1.1.3]{HL85}), we have $\pi_1(V)\cong \pi_1(U)$.
By the assumption (\ref{prop-finite:1}), we have that homomorphisms \[\pi_1(\widehat{C}^o_i)\to\pi_1(V)\to \pe(V)\] is trivial. %, since it factors through the trivial homomorphism $\pi_1(\widehat{Y}^o_i)\to\pi_1(U)\to \pe(U)$. 
 Also, since $C^2=Y^2\cdot H^{n-2}>0$, and $C$ intersects mildly with respect to $B$ by assumption (\ref{prop-finite:2}), we have $H^1(V, \C)=0$ by Corollary \ref{cor:trivialh1}. Hence we can assume $\rho(\pi_1(V))\subset \SL(m,\C)$, since $1$-dimensional representations of $\pi_1(V)$ are finite. We may also assume that $C$ is connected for the rest of the proof.

Now we assume by contradiction that  $\dim \MB(U, m)=\dim \MB(V, m)\neq0$, then there exists a nonrigid semisimple representation $\rho\in \MB(V, m)$.  By the construction in Section \ref{sect:p-adicrep}, we have a pluriharmonic map $u_{\rho}\colon \widetilde{V}\to \Delta(\SL(m, k(A)_{\infty})$ from the universal cover of $V$ to the Bruhat-Tits building $\Delta(\SL(m, k(A)_{\infty})$ associated to an $\infty$-adic field $k(A)_{\infty}$ of an affine curve $A\subset \Hom(\pi_1(V), \SL(m, \C))$. Then we get the spectral covering $\pi\colon T^s\to T$ from a normal surface $T^s$ to $T$ together with $l$ holomorphic $1$-forms $\omega_1, \ldots, \omega_l\in H^0(T^s, \pi^*\Omega^1_{T}(\log D_{\infty}))$ with log poles at infinity (Here $D_{\infty}$ denotes divisors at infinity), and the ramified loci of $\pi$ is contained in the union of the zero loci of $\{\omega_p-\omega_q\}_{1\leq p<q\leq l}$.   %By \cite[Theorem 3.8]{GKP16}, we have that there is normal projective closure $T^s$ of $V^s$ such that $\pi$ extends to a Galois morphism $\bar{\pi}\coloneqq T^s \to T$, i.e., $T\cong T^s/G$ and $\bar{\pi}$ is isomorphic to the quotient map $T^s\to T^s/G$ for $G$ the Galois group of $\pi\colon V^s\to V$. 
Take a logarithmic resolution $\sigma\colon \widehat{T}^{s}\to T^{s}$ such that the boundary divisor $\widehat{T}^{s}-(\pi\circ\sigma)^{-1}V$ is a simple normal crossing divisor, which we denote to be $\widehat{B}$.
For any $1\leq p<q\leq l$, via the morphism $$\sigma^*\circ\pi^*\Omega_{T}^1(\log D_{\infty})\to \Omega_{\widehat{T}^s}^1(\log \widehat{B}),$$ $\omega_p-\omega_q$ gives rise to a nontrivial holomorphic 1-form $\eta_{p,q}$ on $(\pi\circ\sigma)^{-1}V$ with logarithmic poles along boundary divisors, i.e., $\eta_{p,q}\in H^0(\widehat{T}^{s}, \Omega_{\widehat{T}^{s}}^1(\log \widehat{B})).$ For the representation $\rho_{A}\colon \pi_1(V)\to \SL(m, k(A)_{\infty})$, the composition \[\rho_A|_{\widehat{C}^o_i}\colon \pi_1(\widehat{C}^o_i)\to \pi_1(V)\to \SL(m, k(A)_{\infty})\] is the trivial representation, since $\pi_1(\widehat{C}^o_i)\to\pi_1(V)\to \pe(V)$ is the trivial homomorphism and $\SL(m, k(A)_{\infty})$ is residually finite. Also, note that spectral coverings are functorial under the pullback of representations of fundamental groups (See e.g., \cite[Section 3.1]{JZ97} and \cite[Section 4.1.3]{Zuo99} for constructions spectral coverings of $p$-adic representations and Higgs bundles with logarithmic poles, respectively). Hence for each $C_i$ and any irreducible component $C_i^j$ of the scheme theoretic pullback $\pi^*C_i$, either
\begin{enumerate}
\item $\pi$ is \'etale along $C_i^j$, or
\item $\pi|_{C_i^j}$ is nonreduced.
\end{enumerate}
\cite[Lemma 3.2]{Zuo96b} for similar arguments). Hence by \cite[Proposition II.8.10]{Har77} we get that for each connected component $E$ of $\pi^*C$, either $\pi$ is \'etale along $E$, or there exist $1\leq p<q\leq l$ such that $\omega_p-\omega_q$ vanishes along $E$.

For the case (1) that $\pi$ is \'etale along $E$, without loss of generality, we can assume that $\sigma^*E$ shares no common component with the exceptional divisor of $\sigma$, and hence $\sigma^*E$ intersects mildly with respect to $\widehat{B}$, since $E$ does not intersect with the singular loci of $X^s$ and the singular loci of $\pi^{-1}(D)$. Since $\rho_A|_{\widehat{C}^o_i}\colon \pi_1(\widehat{C}^o_i)\to \pi_1(V)\to \SL(m, k(A)_{\infty})$ is the trivial representation for each $i$, $u_{\rho}$ is constant on the preimage of $C$ in $\widetilde{V}$. Hence $\eta_{p,q}|_{\sigma^*E}=0$. For the case (2) that $\omega_p-\omega_q$ vanishes along $E$, the logarithmic 1-form $\eta_{p,q}$ corresponding to $\omega_p-\omega_q$ vanishes along the effective divisor $\sigma^*E$. If $\widehat{B}'$ is the common components of $\sigma^*E$ and $\widehat{B}$,  by \cite[2.3 Properties (c)]{EV92} we have \[\eta_{p, q}\in H^0(\widehat{T}^{s}, \Omega_{\widehat{T}^{s}}^1(\log   \widehat{B}-\widehat{B}')).\]
%To prove the above claim, we may assume $\widehat{B}'$ is not a component of $(\pi\circ\sigma)^*C$. Now consider the following diagram  
%\begin{equation}\label{holo-log-form-compare}
%\begin{tikzpicture}[baseline= (a).base]
%\node[scale=.93] (a) at (0,0){
%\begin{tikzcd}
% H^0(\Omega_{\widehat{T}^{s}}^1(\log \widehat{B}-\widehat{B}')\otimes\mathcal{O}_X(-(\pi\circ\sigma)^*C)) \ar[r]\ar[d]& H^0(\Omega_{\widehat{T}^{s}}^1(\log \widehat{B}-\widehat{B}')) \ar[d]\\
%   H^0(\Omega_{\widehat{T}^{s}}^1(\log \widehat{B})\otimes\mathcal{O}_X(-(\pi\circ\sigma)^*C))\ar[r]& H^0(\Omega_{\widehat{T}^{s}}^1(\log \widehat{B})).
%\end{tikzcd}
%};
%\end{tikzpicture}
%\end{equation}
%The left vertical arrow is an isomorphism. Indeed, notice that by the short exact sequence of residues of logarithmic 1-forms, the cokernel of the left vertical arrow is contained in $H^0(\widehat{B}, \mathcal{O}_{\widehat{B}}\otimes\mathcal{O}_{\widehat{T}^s}(-(\pi\circ\sigma)^*C))$, which is zero since $(\pi\circ\sigma)^*C|_{\widehat{B}}$  has positive degree. Therefore the above claim follows. 
Then we have that $\sigma^*E$ intersects mildly with respect to $\widehat{B}-\widehat{B}'$. Similarly, we have $\eta_{p,q}|_{\sigma^*E}=0$. Notice also there exists a connected component $E$ of $(\pi\circ\sigma)^*C$ such that $E^2>0$, since $((\pi\circ\sigma)^*C)^2=\deg\pi \cdot C^2>0$. This contradicts to Lemma \ref{lem:quasiprospurr}. Hence  we get \[\dim \MB(U, m)=\dim \MB(V, m)=0.\]

For any semisimple representation $\rho$ in the conjugacy class in $M_{\text{B}}(V, m)$, $\rho$ is valued in some number field $K$, since $\rho$ is rigid. Assume that $[\rho(\pi_1(V))\colon \rho(\pi_1(V))\cap \SL(m, \mathcal{O}_K)]=\infty$ where $\mathcal{O}_K$ is the ring of integers  in $K$, then by the construction in Section \ref{sect:p-adicrep} we get the spectral cover $\pi\colon T^s\to T$ from a normal surface $T^s$ to $T$ together with $l$ holomorphic $1$-forms $\omega_1, \ldots, \omega_l\in H^0(T^s, \pi^*\Omega^1_{T}(\log D_{\infty}))$ with log poles along the boundary, and the ramified loci of $\pi$ are contained in the union of the zero loci of $\{\omega_i-\omega_j\}_{1\leq i<j\leq l}$. The argument follows exactly the same as before, which gives a contradiction to Lemma \ref{lem:quasiprospurr}. Hence 
$[\rho(\pi_1(V))\colon \rho(\pi_1(V))\cap \SL(m, \mathcal{O}_K)]<\infty$. This proves the second part of the proposition.
\end{proof}

\subsection{Some group theoretic lemmas} For smooth algebraic varieties, we have the following simple observations.

\begin{lemma}\label{Lemma:etalefiniteimage}
Let $f\colon X^e\to X$ be a finite \'etale morphism. If any finite dimensional complex linear representation of $\pi_1(X^e)$ has finite image, then any finite dimensional complex linear representation of $\pi_1(X)$ also has finite image.
\end{lemma}

\begin{proof}
Take any representation $\rho\colon \pi_1(X)\to \GL(m, \C)$. Since $[\pi_1(X)\colon f_{\ast}(\pi_1(X^e))]<\infty$ and $\rho\circ f_{\ast}(\pi_1(X^e))$ is finite, we have $\rho(\pi_1(X))$ is finite.
\end{proof}
\begin{lemma}\label{trivialimagefund}
Let $U$ be a smooth variety and $Y=\sum_{i=1}^ra_iY_i\subset U$ be a $\pi_1$-small divisor on $U$. Then there is a finite Galois \'etale cover $\tau\colon U^e\to U$ such that for any irreducible divisor $W$ in $\tau^*Y$, the composition \textup{$$\pi_1(\widehat{W})\to\pi_1(U^e)\to\pe(U^e)$$} is trivial, where $\widehat{W}$ is a desingularization of $W$.
\end{lemma}

\begin{proof}
It suffices to show the lemma for a fixed $Y_i$. By the assumption, we have that  $\pi_1(\widehat{Y}_i)\to\pi_1(U)$ is finite for any desingularization $\widehat{Y}_i$. Denote the finite group $\im\{\pi_1(\widehat{Y}_i)\to\pi_1(U)\}$ by $I_i$. Then there is a finite index normal subgroup $G$ of $\pi_1(U)$ such that \[G\cap I_i=\ker \{\pi_1(U)\to\pe(U)\}\cap I_i.\]
If $\pi_1(\widehat{W})\to\pi_1(U^e)\to\pe(U^e)$ is not trivial, then $I_{i}\not\subset \ker \{\pi_1(U)\to\pe(U)\}$. Therefore one can take the Galois \'etale cover $\tau\colon U^e\to U$ corresponding to the above proper normal subgroup $G\lhd\pi_1(U)$. Each irreducible divisor $W$ in $\tau^*Y_{i}$ satisfies the following inequality of cardinalities \[|\im\{\pi_1(\widehat{W})\to\pi_1(U^e)\}|<|I_{i}|,\] where $\widehat{W}$ is a desingularization of $W$. Then by the induction on the cardinality of $I_i$, the lemma holds true.
\end{proof}

\subsection{Proof of the main theorem}
\begin{proof}[Proof of Theorem \ref{Thm:main}]
Without loss of generality, we may assume that $Y$ is an effective divisor. We use the same notations listed in the proof of Proposition \ref{prop:finitebetti}. 
\begin{itemize}
\item $B_i\coloneqq D_i|_T$, $B\coloneqq \sum_{i=1}^sB_i$;
\item $V\coloneqq T-B$;
\item $C^o\coloneqq Y^o|_T=\sum_{i=1}^{r} n_i C^o_i$, where $C^o_i\coloneqq {Y^o_i}|_T$ is irreducible by Bertini theorem;
\item $C\coloneqq Y|_T=\sum_{i=1}^{r} n_i C_i$, where $C_i\coloneqq {Y_i}|_T$;
\item $\widehat{C}^o_i$ are desingularizations of $C_i^o$.
\end{itemize}
Notice that the smooth projective surface $T$ together with the normal crossing boundary divisor $B$ and the divisor $C$ satisfy assumptions (\ref{prop-finite:2}) and (\ref{prop-finite:3}) of Proposition \ref{prop:finitebetti} or Theorem \ref{Thm:main}. Also, the homomorphism \[\pi_1(\widehat{C}^o_i)\to \pi_1(V)\cong\pi_1(U)\] has finite image for each $i$ and some desingularization $\widehat{C}^o_i$ of $C^o_i$.
By Lemma \ref{trivialimagefund}, one can take a finite Galois \'etale cover $\tau\colon V^e\to V$ such that the divisor $\tau^*C^o$ satisfies assumption (\ref{prop-finite:1}) in Proposition \ref{prop:finitebetti}.  By \cite[Theorem 3.8]{GKP16}, $\tau$ can be extended to a finite Galois morphism $\bar{\tau}\colon \overline{V}^e\to T$ where $\overline{V}^e$ is a normal projective surface. Then by \cite[Theorem III.5.2]{BHPV04}, the singular loci $\Sing\overline{V}^e$ of $\overline{V}^e$ are contained in $\bar{\tau}^{-1}(\Sing B)$ and $\bar{\tau}$ is \'etale outside of $\Sing\overline{V}^e$. In particular, $\bar{\tau}^{-1}(C)\cap \Sing \overline{V}^e=\emptyset$ by assumption (\ref{prop-finite:2}) of Proposition \ref{prop:finitebetti} or Theorem \ref{Thm:main}. After resolving the singularities of $\overline{V}^e$, we get a generically finite morphism  $\tau'\colon T'\to T$  such that $\tau'|_{V^e}=\tau$ and $T'$ is a smooth projective surface with $T'-V^e$ being a simple normal crossing divisor. Since $\bar{\tau}^{-1}(C)\cap \Sing \overline{V}^e=\emptyset$, $\tau'^*C$ intersects mildly with respect to $T'-V^e$.
Since $\tau'$ is generically finite, \[(\tau'^*C)^2=\deg \tau'\cdot C^2=\deg\tau'\cdot Y^2\cdot H^{n-2}>0.\] Putting the previous discussions together, and according to Lemma \ref{Lemma:etalefiniteimage}, we may assume  from now on that our original $X, U, D, Y$ satisfy assumptions (\ref{prop-finite:1}), (\ref{prop-finite:2}) and (\ref{prop-finite:3}) of Proposition \ref{prop:finitebetti} and $\dim X=2$. In the following proof, we fix the polarization $H$ and complete the proof by the following steps. 

\subsubsection*{Step 1.} For any representation $\rho\colon \pi_1(U) \to \GL(m, \C)$, we claim that the semisimplification $[\rho]\in\MB(U,m)$ can be deformed to a semisimple representation  in $\MB(U,m)$ such that the corresponding local system underlies a $\C$-PVHS with quasi-unipotent local monodromy along $D$.

In fact, the claim directly follows from  a recently announced result \cite[Theorem 3.2, Corollary 3.3]{EK21}. Here we prove the claim using the method in Mochizuki \cite[Theorem 10.5]{Moc06}. First notice that $\rho$ can be deformed to a semisimple representation $\rho'$ in $\Hom(\pi_1(U), \GL(m, \C))$. The corresponding flat bundle $(V', \mathcal{D}')$ of $\rho'$ admits a tame harmonic metric (Corlette-Jost-Zuo metric \cite{JZ97}). Then the corresponding  parabolic Higgs bundle $(E', \ba E', \theta')$ is a locally abelian $\mu_H$-polystable parabolic Higgs bundle with trivial parabolic Chern classes over $(X, D)$ by Theorem \ref{Thm:K-H-corresp}. According to Proposition \ref{lem:Perturb}, we can deform $(E', \ba E', \theta')$ to another locally abelian $\mu_H$-polystable parabolic Higgs bundle $(E', \baa E', \theta')$ with trivial parabolic Chern classes and all of its jump weight vectors are contained in $\Q^s$.  Next, by the same argument using $\C^*$-action $t\colon (E', \baa E', \theta')\to (E', \baa E', t\cdot\theta')$ in \cite[Theorem 10.5]{Moc06}, we have that $(E', \baa E', \theta')$ deforms to a parabolic Higgs bundle $(E'', \baa E'', \theta'')$ which has the same jump weight vectors as $(E', \baa E', \theta')$ and comes from a $\C$-PVHS. Since the jump weight vectors are rational, $(E'', \baa E'', \theta'')$ comes from a $\C$-PVHS with quasi-unipotent local monodromy along $D$. Moreover, the corresponding local system is semisimple (see  \cite[Proposition 1.13]{Del87}).

\subsubsection*{Step 2.} Reduce to the unipotent case.

By Step 1, for any semisimple representation $\rho\colon \pi_1(U)\to \GL(m, \C)$, the corresponding local system underlies a quasi-unipotent $\C$-PVHS 
$(V, \mathcal{D}, k_V)$ of some weight $w$, since $\dim\MB(U, \C)=0$ by Proposition \ref{prop:finitebetti}. By Kawamata \cite{Ka81},  there is a finite morphism $f\colon X'\to X$ from a smooth projective variety $X'$ with $D' \coloneqq (f^*D)_{\text{red}}$ being a simple normal crossing divisor to $X$, such that $f^*(V, \mathcal{D}, k_V)$ is a $\C$-PVHS with unipotent local monodromy along $D'$. By the construction in Sections \ref{sect:PHBCVHS} and \ref{sect:Semineg}, the pullback $\C$-PVHS $f^*(V, \mathcal{D}, k_V)$ gives rise to a parabolic Higgs bundle $(E'=\oplus_iE'^{i, w-i}, \ba E', \theta')$ with trivial parabolic structure, since it is unipotent. By Proposition \ref{thm:vhs-tame} and Theorem \ref{Thm:K-H-corresp}, $(E'=\oplus_iE'^{i, w-i}, \ba E', \theta')$ is a locally abelian $\mu_{f^*H}$-polystable parabolic Higgs bundle with trivial parabolic Chern classes. As explained in Section \ref{sect:Semineg}, we get a Higgs bundle $(\diae'=\oplus_i \diae'^{i, w-i}, \theta')$ over $X'$ with Higgs field $\theta'\colon \diae'^{i, w-i}\to \Omega_{X'}^1(\log D')\otimes \diae'^{i-1, w-i+1}$. Since the parabolic structure is trivial, we have (see e.g., \cite[Section 3]{AHL19})
\begin{equation}\label{eq:para-normal}
\pc_i((E'=\oplus_iE'^{i, w-i}, \ba E', \theta'))=c_i(\diae')
\end{equation}
Hence we get\begin{equation}\label{eq:degree0}
\deg (\diae')= c_1(\diae')\cdot (f^*H)=\pc_1((E'=\oplus_iE'^{i, w-i}, \ba E', \theta'))\cdot (f^*H)=0
\end{equation}

\subsubsection*{Step 3.} Claim that $\theta'=0$.

For the $\mu_{f^*H}$-polystable Higgs bundle $(\diae'=\oplus_i \diae'^{i, w-i}, \theta')$ over $(X', D')$, assume that $\theta'\neq0$, then there exists a $\mu_{f^*H}$-stable direct summand $(\diaf'=\oplus_i \diaf'^{i, w-i}, \theta')$ of $(\diae'=\oplus_i \diae'^{i, w-i}, \theta')$  such that $\theta'(\diaf'^{i,w-i})\neq0$ for some direct summand $\diaf'^{i,w-i}$ of $\diaf'$. Since $\theta'(\diaf'^{i, w-i})\subset \Omega_{X'}^1(\log D')\otimes \diaf'^{i-1, w-i+1}$, one can take the smallest $i_0$ such that $\diaf'^{i_0,w-i_0}\neq0$ and $\diaf'^{i_0,w-i_0}\subset \ker\theta'$. Since $(\diaf'=\oplus_i \diaf'^{i, w-i}, \theta')$ is a $\mu_{f^*H}$-stable parabolic Higgs bundle with trivial parabolic structure, we have \[\frac{\deg(\diaf'^{i_0,w-i_0})}{\rank(\diaf'^{i_0,w-i_0})}<\frac{\deg(\diaf')}{\rank(\diaf')}=\frac{\deg(\diae')}{\rank(\diae')}=0.\] Hence we get \begin{equation}\label{eq:fhn-1<0}
c_1(\diaf'^{i_0,w-i_0})\cdot (f^*H)<0.
\end{equation}

For the finite morphism $f\colon X'\to X$ in Step 2, we consider $Y'\coloneqq f^*Y$. Note that $Y'^2=\deg f\cdot Y^2>0$, and $Y'$ intersects mildly with respect to $D'$. For each irreducible component $Y'_k$ of $Y'$, we have the following commutative diagram
\begin{equation}\label{diag:cut}
\xymatrix{
\widehat{Y}_k'\ar[d]\ar[r]&Y'_k\ar[d]\ar@{^{(}->}[r]&  U'\coloneqq X'-D'\ar[d]^f \\
\widehat{Y}_i\ar[r] &Y_i\ar@{^{(}->}[r]&  U}
\end{equation} where $Y_i=f(Y'_k)$; $\widehat{Y}_k'$ and $\widehat{Y}_i$ are desingularizations of $Y_k'$ and $Y_i$, respectively. For the representation $\rho\colon \pi_1(U)\to \GL(m, \C)$ in Step 2, the composition induced by the diagram (\ref{diag:cut}) \[\pi_1(\widehat{Y}_k')\to\pi_1(\widehat{Y}_i)\to\pi_1(U)\to \GL(m, \C)\]
is trivial, since $\pi_1(\widehat{Y}_i)\to \pi_1(U)\to \pe(U)$ is trivial for all $i$ by assumption (\ref{prop-finite:1}) of Proposition \ref{prop:finitebetti}. Hence the pullback of the corresponding $\C$-PVHS 
$(V, \mathcal{D}, k_V)$ of $\rho$ to $\widehat{Y}_k'$ is trivial. In particular, we have $\theta'|_{\widehat{Y}_k'}=0$ and \[(\diaf'|_{\widehat{Y}_k'}, \theta'|_{\widehat{Y}_k'})=\bigoplus_i(\diaf'^{i, w-i}|_{\widehat{Y}_k'}, 0).\] Since $\bigoplus_i(\diaf'^{i, w-i}|_{\widehat{Y}_k'}, 0)$ is again a polystable Higgs bundle with trivial Chern classes, \[\deg \diaf'^{i_0, w-i_0}|_{\widehat{Y}_k'}=\deg\diaf'|_{\widehat{Y}_k'}=0.\] This implies \begin{equation}\label{fc=0}
c_1(\diaf'^{i_0,w-i_0})\cdot Y'_k=c_1(\diaf'^{i_0,w-i_0})\cdot Y'=0.
\end{equation}
By Theorem \ref{Thm:semi-neg}, we also have 
\begin{equation}\label{f2hn-2>=0}
c_1(\diaf'^{i_0,w-i_0})^2\geq 0.
\end{equation} 
By equation (\ref{eq:fhn-1<0}), there are positive numbers $m, n>0$ such that $(m  Y'+nc_1(\diaf'^{i_0,w-i_0}))\cdot f^*H=0$. Then $(mY'+nc_1(\diaf'^{i_0,w-i_0}))^2\leq 0$ by the Hodge index theorem. Hence 
$Y'^2\leq0$ by equation (\ref{fc=0}) and (\ref{f2hn-2>=0}), which contradicts to the positivity of $Y'^2$. Thus we have $\theta'=0$.

\subsubsection*{Step 4.} Every linear representation of $\pi_1(U)$ is finite.

Take a general hyperplane curve $L$ of the linear system $|f^*H|$ on $X'$. Denote $A\coloneqq L\cap D'$. Consider the morphisms of log varieties
\[\xymatrix{
(L, A)\ar[r]^i& (X', D')\ar[r]^f&  (X, D)}\]
Since the pullback $\C$-PVHS $f^*(V, \mathcal{D}, k_V)$ corresponds to the parabolic Higgs bundle $(E'=\oplus_iE'^{i, w-i}, \ba E', \theta')$ with $\theta'=0$ by Step 3, the $\C$-PVHS $i^*f^*(V, \mathcal{D}, k_V)$ corresponds to a locally abelian polystable parabolic vector bundle with trivial Chern classes. By \cite{MS80}, the composition representation of $\pi_1(L-A)$
\[\xymatrix{
\pi_1(L-A)\ar[r]^-{i_*}& \pi_1(U')\ar[r]^-{f_*} &\pi_1(U)\ar[r]^-{\rho}&\GL(m, \C)}\]
is unitary. By \cite[Theorem 1.1.3]{HL85} and \cite[Proposition 1.3]{Camp91}, $i_*$ is surjective and $$[\pi_1(U)\colon f_*\pi_1(U')]<\infty.$$ Also, by Proposition \ref{prop:finitebetti}, the representation $\rho\colon \pi_1(U)\to \GL(m, \C)$ is valued in some algebraic integers $\mathcal{O}_K$ (We do further \'etale covers of $U$ at the beginning of the argument if necessary). Hence the image of any semisimple representation $\rho$ of $\pi_1(U)$ is finite. Notice also $H_1(U, \Z)=0$ by Corollary \ref{cor:trivialh1}. Then the finiteness of an arbitrary representation $\eta$ follows exactly the same argument as that in \cite[Page 124]{Zuo99}. More specifically, by Lemma \ref{Lemma:etalefiniteimage} we can take the \'etale cover $\tau\colon U^e\to U$ such that the image of the representation $\eta\circ\tau_*$ is contained a unipotent subgroup. Then via analyzing the composition series of the unipotent subgroup, we get that $\eta$ has finite image, since $H_1(U^e, \Z)=0$.
\end{proof}

By the above proof, we also have the following \begin{corollary}\label{cor:strong-main}
Let $X$ be a smooth projective variety with a normal crossing divisor $D$ and a divisor $Y=\sum_{i=1}^r n_iY_i$. For $U\coloneqq X-D$ and $Y^o\coloneqq Y|_U$,  suppose that 
\begin{enumerate}
\item \label{cor:main1} The homomorphism \textup{\[\pi_1(\widehat{Y}^o_i)\to \pi_1(U)\to\pe(U)\]} is trivial for each $i$ and some desingularization $\widehat{Y}^o_i$ of $Y^o_i$,

\item \label{cor:main2} $Y$ intersects mildly with respect to $D$, and

\item \label{cor:main3} There is an ample divisor $H$ on $X$, such that $Y^2\cdot H^{n-2}>0$.
\end{enumerate}
Then any linear representation of the fundamental group $\rho\colon\pi_1(U)\to \GL(m, \C)$ is finite. 
\end{corollary}

To end this section, we prove Corollary \ref{Cor:main-proj} and Corollary \ref{Cor:infiniteneg} in the introduction.

\begin{proof}[Proof of Corollary \ref{Cor:main-proj}]
Consider the resolution of singularities $\mu\colon X\reso\to X$. By \cite{Tak03}, $\pi_1(X\reso)\cong\pi_1(X)$. Notice that $\mu^*Y$ is a $\pi_1$-small divisor on $X\reso$. For the ample line bundle $H$, $\mu^*H$ is big and nef. Hence there is an effective divisor $B$ and an ample $\Q$-divisor $H_k$ on $X\reso$ such that $\mu^*H\sim_{\text{num}}H_k+\frac{1}{k}B$ for $k\gg 1$. Thus we have $$(\mu^*Y)^2\cdot H_k^{n-2}=(\mu^*Y)^2\cdot(\mu^*H-\frac{1}{k}B)^{n-2}.$$ One can take a large enough $k$ such that $(\mu^*Y)^2\cdot H_k^{n-2}>0$. Then the corollary follows from Theorem \ref{Thm:main}.
\end{proof}

\begin{proof}[Proof of Corollary \ref{Cor:infiniteneg}]
By \cite[Lemma 4.1]{MVZ09} and the assumption of Corollary \ref{Cor:infiniteneg}, there is a divisor $D\coloneqq \sum_ia_iY_i$ with $a_i\in\N$ such that $D^2\cdot H^{n-2}>0$. Then the corollary follows directly from Theorem \ref{Thm:main}.
\end{proof}

\section{$\pi_1$-small curves and hyperbolicity properties of surfaces}

In this section, we prove Proposition  \ref{C}  and Theorem \ref{Thm:main-hyperbolicity} in the introduction.

\begin{proof}[Proof of Proposition \ref{C}]
Consider the Shafarevich map of $X$ $$\Sh_X\colon X\dashrightarrow \Sh(X).$$ Since $\pi_1(X)$ is infinite, $\dim \Sh(X)>0$ by Theorem \ref{Thm:Shafarevich}. For the set of $\pi_1$-small curves $\{C_i\}_{i\in\Lambda}$, we need to show that if one of the following two conditions holds true \begin{itemize}
\item The statement (\ref{glf1}) does not hold, and $C_i^2>0$ for $i\in\Lambda$,
\item $C_i^2<0$ for $i\in\Lambda$,
\end{itemize}
then $X$ has generically large fundamental group, i.e., $\dim \Sh(X)=2$. We assume by contradiction that $\dim \Sh(X)=1$.
Take blow-ups of $X$ to resolve the indeterminancy of $\Sh_X$ and denote $T$ to be the set of blow-up centers in this procedure. Since one of the above two conditions holds, there is an irreducible $\pi_1$-small divisor $D\in\{C_i\}_{i\in\Lambda}$ such that $\Sh_X(D)$ is dense in $\Sh(X)$. Do more blow-ups of $X$ to resolve the singularities of $D$, we get the following commutative diagram
$$
\xymatrix{
\widehat{D}\ar[d]\ar@{^{(}->}[r]^-{\widehat{i}}& \widehat{X} \ar[d]^{\eta}\ar[r]^-{\widehat{\Sh}_X}&\Sh(X)\ar@{=}[d] \\
D \ar@{^{(}->}[r]^-{i}& X\ar@{-->}[r]^-{\Sh_X}&\Sh(X), }
$$
where $\eta$ is the composition of blow-ups, $\widehat{D}$ is the smooth strict transform of $D$, and $i$ and $\widehat{i}$ are natural inclusions. Since $\widehat{\Sh}_X$ is birational to $\Sh_X$, the general fibre of $\widehat{\Sh}_X$ is $\pi_1$-small. Let $S$ be the set of points on $\Sh(X)$ over which $\widehat{\Sh}_X$ is not smooth. Denote $\widehat{X}^o\coloneqq \widehat{\Sh}_X^{-1}(\Sh(X)-S)$, and the general fibre of $\widehat{\Sh}_X$ by $F$. Then we have the following diagram of fundamental groups$$
\xymatrix{
 \pi_1(F) \ar[r]&\pi_1(\widehat{X}^o)\ar[r]\ar@{->>}[d]^-{\phi}&\pi_1(\Sh(X)-S)\to \{1\}\\
&\pi_1(\widehat{X}), }
$$
where $\phi$ is surjective. Then the image of $\pi_1(F)$ is a normal subgroup of $\pi_1(\widehat{X}^o)$, which we denote as $N$. Hence $\phi(N)\lhd\pi_1(\widehat{X})$ is a finite normal subgroup. Since $\widehat{D}$ is $\pi_1$-small, the subgroup generated by $\phi(N)$ and $\widehat{i}_*(\pi_1(\widehat{D}))$ in $\pi_1(\widehat{X})$ is finite. However this contradicts to \cite[Proposition 1.4]{Camp91}, since $\pi_1(\widehat{X})\cong\pi_1(X)$ is an infinite group. 
\end{proof}

Next we use a similar method as in the proof of the above property together with \cite{Zuo96b} to show Theorem \ref{Thm:main-hyperbolicity}.  First we recall the following notations for algebraic groups (see e.g., \cite[ Section 2.1]{Zuo99}). A group $G$ is called \textsl{almost abelian} if $G$ contains a finite index abelian subgroup.  An algebraic group $G\subset \GL(n, \C)$ is called \textsl{almost simple} if all the proper normal algebraic subgroups are finite groups. Here $G$ is the group of complex points of the corresponding group scheme. We have the following group theoretic lemma.
\begin{lemma}\label{lem:almost-simple}
Any non-commutative almost simple complex algebraic group does not contain an infinite Zariski dense almost abelian group as a subgroup.
\end{lemma}

\begin{proof}
Assume that there is an almost simple algebraic group $G$ containing an infinite Zariski dense almost abelian subgroup $H$. Then $\dim_{\C} G>0$. Let $A$ be an infinite abelian subgroup of $H$ such that $H/A$ is finite. Hence the Zariski closure $\overline{A}$ in $G$ is a finite index abelian subgroup of $G$. This is impossible.
\end{proof}

\begin{proof}[Proof of Theorem \ref{Thm:main-hyperbolicity}]
First note that (\ref{Thm:rat-curve}) follows immediately from Theorem \ref{Thm:L-R} and Miyaoka-Yau-Sakai inequality. In fact, by the assumption and Theorem \ref{Thm:L-R}, any rational curve $C$ on $X$ satisfies that $C^2\leq0$. Apply the orbibundle Miyaoka-Yau-Sakai inequality \cite[Theorem 1.3 (i)]{Miy08} to the rational curve $C$ on $X$, we have that for any real number $\alpha$ with $0\leq \alpha\leq 1$, the inequality
$$\frac{\alpha^2}{2}(C^2 +3K_{X}\cdot C+6)-2\alpha(K_{X}\cdot C+3)+3c_2(X)-K_{X}^2 \geq0
$$
holds, where $K_{X}$ is the canonical line bundle, and $c_2(X)$ is the topological Euler characteristic. Choose $\alpha=\frac{1}{3}$, we get a uniform bound of canonical degree of rational curves $C$ in terms of  topological invariants of $X$ $$K_{X}\cdot C\leq 6c_2(X)-2K_{X}^2-\frac{5}{3}.$$ Then rational curves form a bounded family on $X$. Since $X$ is of general type, there are only finitely many of rational curves on $X$.  

For (\ref{Thm:ell-curve}), %note that given a finite \'etale morphism $\tau\colon X'\to X$, $X'$ admits at most finitely many curves with geometric genus less than 2 if and only if so does $X$. C
consider the case that the irregularity $q(X)>1$, then by \cite[Theorem 1.1]{LP12} the canonical degrees of geometric genus $1$ curves $C$ of $X$ satisfy \[K_{X}\cdot C\leq \max\{K_{X}^2, 4K_{X}^2+2g-2\}.\] Hence $X$ admits at most finitely many geometric genus 1 curves. To finish the proof of (\ref{Thm:ell-curve}), we may assume that $q(X)=b_1(X)=0$. Also, we may assume that the representation $\rho$ is semisimple. In fact, if the images of all finite dimensional semisimple representations are finite, then any finite dimensional linear representation is finite by the argument in Step 4 of Theorem \ref{Thm:main}.  

Now consider the Zariski closure $\overline{\rho(\pi_1(X))}$ of the reductive group $\rho(\pi_1(X))$ in $\GL(m, \C)$. $\overline{\rho(\pi_1(X))}$ is isogeny to a direct product $T\times\Pi_{i=0}^dG_i$, where $T$ is a complex torus $(\C^*)^l$ and $G_i$ are non-commutative almost simple algebraic groups. Since $b_1(X)=0$, $T=0$. Thus we can pick an almost simple factor, say $G_0$, such that the induced representation $\rho_0\colon \pi_1(X)\to G_0$, via the projection to $G_0$, is a Zariski dense representation into a non-commutative almost simple group $G_0$ with infinite image $\rho_0(\pi_1(X))$. By Theorem \ref{Thm:Shafarevich}, we have the Shafarevich map associated to the representation $\rho_0$ \[\Sh^{\rho_0}_X\colon X\dashrightarrow \Sh^{\rho_0}(X).\] Since $\rho_0(\pi_1(X))$ is infinite, $\dim \Sh^{\rho_0}(X)>0$. 

Suppose that $\dim \Sh^{\rho_0}(X)=1$. Assume that there are infinitely many geometric genus $1$ curves on $X$. Since $X$ is of general type, general fibres of $\Sh_X^{\rho_0}$ are curves of geometric genus $g>1$. Hence there exists a geometric genus $1$ curve $C$ such that $\Sh^{\rho_0}_X(C)$ is dense in $\Sh^{\rho_0}(X)$. Since $q(X)=0$, we have $\Sh^{\rho_0}(X)\cong \mathbb{P}^1$. Taking blow-ups of $X$ to resolve the singularities of $C$ and the indeterminancy of $\Sh^{\rho_0}_X$, we get a smooth projective surface $\widehat{X}$ and the following commutative diagram
\begin{equation}\label{diag:shafa1}
\xymatrix{
\widehat{C}\ar[d]\ar@{^{(}->}[r]^-{\widehat{i}}& \widehat{X} \ar[d]^{\eta}\ar[r]^-{\widehat{\Sh}^{\rho_0}_X}&\mathbb{P}^1\ar@{=}[d] \\
C \ar@{^{(}->}[r]^-i& X\ar@{-->}[r]^-{\Sh^{\rho_0}_X}&\mathbb{P}^1,}
\end{equation}
where $\eta$ is the blow-up map, and $\widehat{C}$ is the smooth strict transform of $C$. Also, notice that the general fibre of $\Sh^{\rho_0}_X$ is the same as the general fibre of $\widehat{\Sh}^{\rho_0}_X$, which is denoted by $F$. Let $S\subset \mathbb{P}^1$ be the set of points over which $\widehat{\Sh}^{\rho_0}_X$ is not smooth. Denote $\widehat{X}^o\coloneqq (\widehat{\Sh}^{\rho_0}_X)^{-1}(\mathbb{P}^1-S)$. We have the following exact sequence $$
\xymatrix{
 \pi_1(F) \ar[r]&\pi_1(\widehat{X}^o)\ar[r]\ar@{->>}[d]^-{p}&\pi_1(\mathbb{P}^1-S)\to \{1\}\\
&\pi_1(\widehat{X}), }
$$
Then $\im\{\pi_1(F)\to\pi_1(\widehat{X}^o)\}\lhd\pi_1(\widehat{X}^o)$ is a normal subgroup, which we denote as $N$. Hence the image $p(N)\lhd\pi_1(\widehat{X})$ is a normal subgroup. Then the subgroup generated by $p(N)$ and $\widehat{i}_*(\pi_1(\widehat{C}))$ is \[\langle p(N), \widehat{i}_*(\pi_1(\widehat{C}))\rangle=\{g\cdot h\mid\ g\in p(N), h\in\widehat{i}_*(\pi_1(\widehat{C}))\}\] where $\widehat{i}$ is the inclusion in diagram (\ref{diag:shafa1}). Since $F$ is contracted by $\Sh^{\rho_0}_X$, for the representation $\rho_0\colon\pi_1(\widehat{X})=\pi_1(X)\to G_0$, $\rho_0\circ p(N)$ is finite. Also, since $\langle p(N), \widehat{i}_*(\pi_1(\widehat{C}))\rangle$ is a finite index subgroup of $\pi_1(\widehat{X})$ by \cite[Proposition 1.4]{Camp91}, we have that $\rho_0\circ \widehat{i}_*(\pi_1(\widehat{C}))$ is a finite index infinite abelian subgroup of $\rho_0(\pi_1(\widehat{X}))$ (Here $\widehat{C}$ is a smooth elliptic curve). Hence $G_0$ contains the Zariski dense subgroup $\rho_0(\pi_1(\widehat{X}))$, which is an infinite almost abelian group. This contradicts to Lemma \ref{lem:almost-simple}. Hence in the case that $\dim \Sh^{\rho_0}(X)=1$, $X$ contains at most finitely many geometric genus $1$ curves. When $\dim \Sh^{\rho_0}(X)=2$, we have a Zariski dense representation of $\pi_1(X)$ to an almost simple algebraic group $G_0$ such that the associated Shafarevich map is birational. Again $X$ has only finitely many geometric genus $1$ curves by \cite[Theorem 2]{Zuo96b}.
\end{proof}

\end{document}